\documentclass[12pt]{amsart}

\usepackage{graphicx}
\usepackage{amssymb}
\usepackage{float}
\usepackage{fullpage}

\newtheorem{theorem}{Theorem}[section]
\newtheorem{lemma}[theorem]{Lemma}
\newtheorem{corollary}[theorem]{Corollary}
\newtheorem{proposition}[theorem]{Proposition}

\theoremstyle{definition}

\theoremstyle{remark}

\numberwithin{equation}{section}

\newcommand{\N}{{\mathbb N}}

\begin{document}

\title[Almost primes]{Almost primes and the Banks--Martin conjecture}

\author{Jared Duker Lichtman}
\address{Mathematical Institute, University of Oxford, Oxford, OX2 6GG, UK}

\email{jared.d.lichtman@gmail.com}


\date{December 18, 2019.}



\begin{abstract}
It has been known since Erd\H os that the sum of $1/(n\log n)$ over
numbers $n$ with exactly $k$ prime factors (with repetition)
is bounded as $k$ varies.  We prove that as $k$ tends to infinity,
this sum tends to 1.  Banks and Martin have conjectured that
these sums decrease monotonically in $k$, and in earlier
papers this has been shown to hold for $k$ up to 3. However, we show
that the conjecture is false in general, and in fact a global minimum occurs at $k=6$.
\end{abstract}

\maketitle

\section{Introduction}

Let $\Omega(n)$ denote the number of prime factors of $n$, counted with repetition. For each $k\ge1$, let $\N_k=\{n:\Omega(n)=k\}$ be the set of $k$-almost primes. The sets $\N_k$ are the prototypical examples of {\it primitive} sets of natural numbers $>1$, i.e., no member of the set divides any other. Erd\H os \cite{Erd} proved that $f(A) := \sum_{n\in A}1/(n\log n)$ is bounded uniformly over all primitive sets $A$. Moreover in 1988, he conjectured that $f(A)\le f(\N_1)=1.636\cdots$ for any primitive set $A$. The current record is $f(A)<e^\gamma=1.781\cdots$, due to Lichtman and Pomerance \cite{LPprim} \cite{LPrace}.

In 2013, Banks and Martin \cite{BM} conjectured that $f(\N_k)>f(\N_{k+1})$ for all $k\ge1$, that is,
\begin{align*}
\sum_{n\in\N_1}\frac{1}{n\log n} > \sum_{n\in\N_2}\frac{1}{n\log n} > \sum_{n\in\N_3}\frac{1}{n\log n} > \cdots.
\end{align*}
Their conjecture may be considered as an extension of Erd\H os', exemplifying the general view that $f(A)$ is sensitive to the prime factorizations of $n\in A$. Indeed, Banks and Martin \cite{BM} showed that, for any sufficiently small set of primes $\mathcal Q$ (e.g., if  $\sum_{p\in\mathcal Q}1/p < 1.74$), the analogous statement holds,
\begin{align}
    f\big(\N_k(\mathcal Q)\big)  >  f\big(\N_{k+1}(\mathcal Q)\big)\qquad\textrm{for all } k\ge1.
\end{align}
Here $A(\mathcal Q)$ denotes the numbers in $A$ composed only of prime factors in $\mathcal Q$. In fact they showed $f\big(\N_k(\mathcal Q)\big) \ge f(A(\mathcal Q))$ for all primitive sets $A$ with $\Omega(n)\ge k$ for each $n\in A$.

Zhang \cite{Zh2} proved that $f(\N_1) > f(\N_k)$ for each $k\ge2$. Lichtman and Pomerance \cite{LPprim} proved that $f(\N_k) \gg 1$. Bayless, Kinlaw, and Klyve \cite{BKK} recently showed that $f(\N_2) > f(\N_3)$, providing bounds on $f(\N_k),f(\N^*_k)$ for small $k$, see \ref{fig:BKK}. Here $\N_k^*$ is the set of {\it squarefree} $k$-almost primes.

\begin{figure}[H]
  \caption{Bounds on $f(\N_k),f(\N^*_k)$ from Bayless et al. \cite{BKK}} \label{fig:BKK}
\begin{align*}
\begin{tabular}{c|l}
$k$ & $f(\N_k)$ \\
\hline 
2 & (1.1416, 1.1484)\\
3 & (0.65708, 1.0841)\\
4 & (0.40713, 1.1891)
\end{tabular}
\begin{tabular}{c|l}
$k$ & $f(\N^*_k)$ \\
\hline
2 & (0.8877, 0.8945)\\
3 & (0.36003, 0.7678)\\
4 & (0.15118, 0.8527)
\end{tabular}
\end{align*}
\end{figure}

Their approach is to directly compute the series up to $10^{12}$, and then obtain explicit inequalities for the counting functions of $\N_k,\N^*_k$. By partial summation, these translate into bounds for $f(\N_k),f(\N^*_k)$. As evidenced by the table, this approach becomes exceedingly difficult as $k$ grows. This is in part due to the fact that the series for $f(\N_k),f(\N^*_k)$ converge quite slowly. For example, as we shall see, the partial sum up to $10^{12}$ makes up less than half of $f(\N_4)$.

Even in the case $k=1$ for primes, the series for $f(\N_1)$ converges slowly. Nevertheless, Cohen \cite{Cohen} was able to compute to a remarkable degree of precision,
\begin{align}
f(\N_1) \ = \ 1.63661632335126086856965800392186367118159707613129\cdots.
\end{align}
His basic idea is to write
\begin{align*}
f(\N_1) = \sum_p \frac{1}{p\log p} =\int_1^\infty P(s)\;ds,
\end{align*}
where $P(s) = \sum_p p^{-s}$ is the prime zeta function. Since $\log\zeta(s)=\sum_{m\ge1}P(ms)/m$, by M\"obius inversion one has the rapidly converging series $P(s) = \sum_{m\ge1}(\mu(m)/m)\log\zeta(ms)$, and in turn one may use well-known rapid computation of $\zeta(s)$.

\section{Statement of results}

First, by extending the zeta function method initiated by Cohen we generate the following data in Figure \ref{fig:fNk}, refining Figure \ref{fig:BKK} from \cite{BKK}.

\begin{figure}[H]
  \caption{Computation of $f(\N_k),f(\N^*_k)$ to 20 digits.} \label{fig:fNk}
\begin{align*}
\begin{tabular}{c|l}
$k$ & $f(\N_k)$ \\
\hline 
2 & 1.1448165734059179915\\
3 & 1.0308351017932175719\\
4 & 0.9973421485952523597\\
5 & 0.9888821921300755349\\
6 & 0.9887534530145096063\\
7 & 0.9910205950027380022\\
8 & 0.9935373386530404095\\
9 & 0.9956203792390954090\\
10& 0.9971495172651382446
\end{tabular}
\begin{tabular}{c|l}
$k$ & $f(\N^*_k)$ \\
\hline
2 & 0.8909254794763183321\\
3 & 0.7131238005098902554\\
4 & 0.6528129098554062569\\
5 & 0.6284306642973934048\\
6 & 0.6176406880308143497\\
7 & 0.6126252367925047050\\
8 & 0.6102275665474058560\\
9 & 0.6090620642567092069\\
10 & 0.6084897027941833669
\end{tabular}
\end{align*}
\end{figure}
We observe that the sequence $\{f(\N_k)\}_k$ decreases for $k\le 6$ but then increases thereafter, in particular $f(\N_6)<f(\N_7)$, contrary to the conjecture of Banks--Martin.

We also prove the following theorems, which confirm some of the trends observed in the data.
\begin{theorem}\label{thm:globmin}
For all positive integers $k\neq6$, we have $f(\N_6) < f(\N_k)$.
\end{theorem}
\begin{theorem}\label{thm:fNk1}
In the limit as $k\to\infty$, we have
\begin{align}
f(\N_k) \ = \sum_{\Omega(n)=k}\frac{1}{n\log n} \ \sim \ 1 \  \quad \text{and}\quad f(\N^*_k) \ = \sum_{\Omega(n)=k}\frac{\mu(n)^2}{n\log n} \ \sim \ \frac{6}{\pi^2}.
\end{align}
\end{theorem}

We shall prove a quantitative form of Theorem \ref{thm:fNk1} in Section \ref{sec:largek}, namely, $f(\N_k)=1+O_\epsilon(k^{\epsilon-1/2})$ for any $\epsilon>0$. The proof uses partial summation and a delicate application of the Sathe-Selberg theorem to $N_k(x) := \#\{n\le x:\Omega(n)=k\}$ in the critical range $x \approx e^{e^k}$.

\begin{theorem}[Sathe-Selberg \cite{SathSelb}]
For $k\le 1.99\log\log x$,
\begin{align*}
N_{k+1}(x) \ = \ G\Big(\frac{k}{\log\log x}\Big)\frac{x}{\log x}\frac{(\log\log x)^k}{k!}\bigg(1 \ + \ O\Big(\frac{k}{(\log\log x)^2}\Big)\bigg),
\end{align*}
where $G(z) = \frac{1}{\Gamma(1+z)}\prod_p(1-z/p)^{-1}(1-1/p)^z$.
\end{theorem}

This approach is sufficiently versatile to generalize to arithmetic progressions. However, as we shall see throughout the paper, the zeta function method offers a much more precise understanding of $f(\N_k)$. By this method, we shall prove Theorem \ref{thm:globmin} as well as obtain an {\bf exponentially decaying} upper bound of $1 + O(k\,2^{-k/2})$, which dramatically refines the results by classical methods. 
Further, we expect zeta functions may lead to lower bounds of comparable strength, though the problem here is more subtle.

In the next section, we describe the zeta function method and show how to generate the data in Figure \ref{fig:fNk} for small $k$.

\section{Zeta function method for small $k$}

Consider the prime and $k$-almost prime zeta functions
\begin{align*}
P(s) = \sum_p \frac{1}{p^s} \qquad P_k(s) = \sum_{\Omega(n)=k} \frac{1}{n^s} \qquad P^*_k(s) = \sum_{\Omega(n)=k} \frac{\mu(n)^2}{n^s}
\end{align*}
for $k\ge1$. Note $P_1=P_1^*=P$, and let $P_0(s)=1$. Our interest in these zeta functions arises from the identity
\begin{align}\label{eq:ftozeta}
f(\N_k) = \sum_{\Omega(n)=k}\frac{1}{n\log n} =  \sum_{\Omega(n)=k}\int_1^\infty\frac{ds}{n^s} = \int_1^\infty P_k(s)\;ds
\end{align}
Here $f(\N_k)$ is finite, and so by Tonelli's theorem the interchange of sum and integral in \eqref{eq:ftozeta} is justified because all terms are positive. Similarly we have $f(\N^*_k) = \int_1^\infty P^*_k(s)\;ds$.

In the following proposition, we express $P_k$ explicitly in terms of $P$ and derive a handy recursion formula.\footnote{This was stated with incomplete proof in a preprint \cite{Mathar}. We provide a full proof by a simpler method.}
\begin{proposition}\label{prop:Pkpart}
We have
\begin{align}\label{eq:Pkpart}
P_k(s) = \sum_{n_1 + 2n_2 +\cdots + kn_k = k} \prod_{j=1}^k \frac{1}{n_j!}\big(P(js)/j\big)^{n_j}
\end{align}
where the sum ranges over all partitions of $k$. $P_k$ satisfies the recurrence
\begin{align}\label{eq:Pkrecur}
P_k(s) = \frac{1}{k}\sum_{j=1}^k P(js)P_{k-j}(s).
\end{align}
Similarly, $P_k^*$ satisfies
\begin{align}\label{eq:Pksq}
P^*_k(s) \ = \ (-1)^k\sum_{n_1 + 2n_2 +\cdots = k} \prod_{j=1}^k \frac{1}{n_j!}\big(-P(js)/j\big)^{n_j} \ = \ \frac{1}{k}\sum_{j=1}^k (-1)^{j+1}P(js)P^*_{k-j}(s).
\end{align}
\end{proposition}
\begin{proof}
For \eqref{eq:Pkpart}, we have the following identity for $\zeta(s)$,
\begin{align*}
\sum_{k\ge0}P_k(s) = \zeta(s) = \prod_p(1-p^{-s})^{-1} = \exp\Big(\sum_{j\ge1}\frac{P(js)}{j}\Big).
\end{align*}
Similarly, we have a formal power series identity in $z$,
\begin{align}
\sum_{k\ge0}P_k(s)z^k &= \prod_p\big(1 - zp^{-s}\big)^{-1} = \exp\Big(\sum_{j\ge1}\frac{P(js)}{j}z^j\Big) = \prod_{j\ge1}\exp\Big(\frac{P(js)}{j}z^j\Big) \nonumber\\
& =\prod_{j\ge1}\sum_{n_j\ge0}\frac{1}{n_j!}\Big(\frac{P(js)}{j}z^j\Big)^{n_j} 
= \sum_{k\ge0}z^k \ \sum_{n_1 + 2n_2 +\cdots = k} \prod_{j\ge1} \frac{1}{n_j!}\big(P(js)/j\big)^{n_j}.
\end{align}
Now \eqref{eq:Pkpart} follows by comparing the coefficients of $z^k$.

To show \eqref{eq:Pkrecur}, each term $n^{-s}$ appearing in $\sum_{j=1}^k P(js)P_{k-j}(s)$ has $\Omega(n)=k$. Conversely, for each $\Omega(n)=k$ write $n=\prod_p p^{e_p}$ with $\sum_p e_p=k$. Then for $1\le j\le k$, the term $n^{-s}$ appears $\#\{p:e_p\ge j\}$ times in $P(js)P_{k-j}(s)$, thus appearing
\begin{align*}
\sum_{j=1}^k\#\{p:e_p\ge j\} = \sum_p e_p = k
\end{align*}
times in $\sum_{j=1}^k P(js)P_{k-j}(s)$. Hence \eqref{eq:Pkrecur} follows.

We proceed similarly in the squarefree case. Letting $P_k^\mu(s) = \sum_{\Omega(n)=k}\mu(n)/n^s = (-1)^kP_k^*(s)$,
\begin{align*}
\sum_{k\ge0} P_k^\mu(s)z^k = \prod_p(1-zp^{-s}) & = \exp\Big(-\sum_{j\ge1} \frac{P(js)}{j} z^j\Big)
= \sum_{k\ge0}z^k \ \sum_{n_1 + 2n_2 +\cdots = k} \prod_{j\ge1} \frac{\big(-P(js)/j\big)^{n_j}}{n_j!}
\end{align*}
by expanding $\exp$ as before. Now the the partition formula in \eqref{eq:Pksq} follows by comparing the coefficients of $z^k$, and using $P_k^\mu(s) = (-1)^kP_k^*(s)$. Finally, as before, by counting the number of appearances of $\mu(n)n^{-s}$ on each side (even when $\mu(n)=0$), we get
\begin{align*}
P_k^\mu(s) \ = \ \frac{1}{k}\sum_{j=1}^k -P(js)\cdot P^\mu_{k-j}(s).
\end{align*}
Using $P_k^\mu(s) = (-1)^kP_k^*(s)$ this reduces to \eqref{eq:Pksq} as claimed.
\end{proof}

For instance, the first few $P_k$ are given by
\begin{align*}
2!\cdot P_2(s) & = P(s)^2 + P(2s)\\
3!\cdot P_3(s) & = P(s)^3 + 3P(2s)P(s) + 2P(3s)\\
4!\cdot P_4(s) & = P(s)^4 + 6P(2s)P(s)^2 + 3P(2s)^2 + 8P(3s)P(s) + 6P(4s)\\
5!\cdot P_5(s) & = P(s)^5+10 P(s)^3 P(2s)+15 P(s) P(2s)^2+20 P(s)^2 P(3s)\\
& \quad +20 P(2s) P(3s)+30 P(s) P(4s)+24 P(5s). 
\end{align*}

With the above expressions for $P_k, P_k^*$ in terms of $P$, and using the built-in function PrimeZetaP, a few lines of code in Mathematica computes $f(\N_k),f(\N^*_k)$ to high precision, generating the data in Figure \ref{fig:fNk}.
\footnote{
{\tiny\tt P[k\_Integer,s\_]:= If[k==1,PrimeZetaP[s], Expand[(Sum[P[1,j*s]*P[k-j,s],{j,1,k-1}]+P[1,k*s])/k]]

Do[Print[k," ",NIntegrate[P[k,s],{s,1,Infinity}, WorkingPrecision->30, AccuracyGoal -> 13,PrecisionGoal -> 13]], {k, 10}]}}

The computation was also independently verified to 20 digits on Pari/GP, courtesy of Paul Kinlaw. Pari/GP does not have $P(s)$ built-in, so it was computed using a variant of the identity
$P(s)=\sum_{m\ge1} \frac{\mu(m)}{m}\log\zeta(ms)$, namely,
\begin{align}
P(s) \ = \ \sum_{p\le A}p^{-s} \ + \ \sum_{m\ge1} \frac{\mu(m)}{m}\log\zeta_A(ms)
\end{align}
for a suitable choice of $A$, where $\zeta_A(s) = \zeta(s)\prod_{p\le A}(1-p^{-s})$.

The data in Figure \ref{fig:fNk} suggest that $f(\N_k)$ tends to 1 and $f(\N^*_k)$ tends to $6/\pi^2\approx .607\cdots$ as $k$ grows. With a bit more patience, we may calculate these differences for $k\le 20$.

\begin{figure}[H]
  \caption{Further computations of $1-f(\N_k)$ and $f(\N^*_k)-6/\pi^2$} \label{fig:fNk2}
\begin{align*}
\begin{tabular}{c|l}
$k$ & $1-f(\N_k)$ \\
\hline 
10& $2.85\cdots\times 10^{-3}$\\
11& $1.80\cdots\times 10^{-3}$\\
12& $1.11\cdots\times 10^{-3}$\\
13& $6.74\cdots\times 10^{-4}$\\
14& $4.02\cdots\times 10^{-4}$\\
15& $2.37\cdots\times 10^{-4}$\\
16& $1.38\cdots\times 10^{-4}$\\
17& $7.96\cdots\times 10^{-5}$\\
18& $4.55\cdots\times 10^{-5}$\\
19& $2.58\cdots\times 10^{-5}$\\
20& $1.45\cdots\times 10^{-5}$
\end{tabular}
\begin{tabular}{c|l}
$k$ & $f(\N^*_k)-6/\pi^2$ \\
\hline
10& $5.62\cdots\times 10^{-4}$\\
11& $2.79\cdots\times 10^{-4}$\\
12& $1.39\cdots\times 10^{-4}$\\
13& $6.95\cdots\times 10^{-5}$\\
14& $3.46\cdots\times 10^{-5}$\\
15& $1.73\cdots\times 10^{-5}$\\
16& $8.65\cdots\times 10^{-6}$\\
17& $4.32\cdots\times 10^{-6}$\\
18& $2.16\cdots\times 10^{-6}$\\
19& $1.08\cdots\times 10^{-6}$\\
20& $5.40\cdots\times 10^{-7}$
\end{tabular}
\end{align*}
\end{figure}


\section{Asymptotic behavior for large $k$} \label{sec:largek}

We confirm the limits that the data suggest with the following theorem.
\begin{theorem} \label{thm:fNksim1}
For any $\epsilon>0$,
\begin{align}
f(\N_k) \ = \ 1 \ + \ O\big(k^{-1/2+\epsilon}\big) \quad \text{and}\quad f(\N^*_k) \ = \ \frac{6}{\pi^2} \ + \ O\big(k^{-1/2+\epsilon}\big).
\end{align}
\end{theorem}
\begin{proof}
Fix $k$ large and let $N_k(x) = \#\{n\le x:\Omega(n)=k\}$. First, by partial summation
\begin{align*}
f(\N_{k+1}) = \sum_{\Omega(n)=k+1}f(n) = -\int_{2^{k+1}}^\infty N_{k+1}(t)f'(t)\;dt = \int_{2^{k+1}}^\infty \frac{N_{k+1}(t)}{t^2\log t}\Big(1 + \frac{1}{\log t}\Big)\;dt =: I.
\end{align*}
The Sathe--Selberg theorem \cite{SathSelb} implies that for $r=1.99$ and $k\le r\log\log x$ (i.e., $x\ge e^{e^{k/r}}$)
\begin{align}
N_{k+1}(x) \ = \ G\Big(\frac{k}{\log\log x}\Big)\frac{x}{\log x}\frac{(\log\log x)^k}{k!}\bigg(1 \ + \ O\Big(\frac{k}{(\log\log x)^2}\Big)\bigg),
\end{align}
where $G(z) = \frac{1}{\Gamma(1+z)}\prod_p(1-z/p)^{-1}(1-1/p)^z$.
As such, we split up the integral $I = I_1 + I_2$ at $t = e^{e^{k/r}}$. For $I_1$, we use the universal bound of Erd\H os-S\'ark\"ozy \cite{ESark},
\begin{align}
N_{k+1}(t) \ \ll \  \frac{k^4}{2^k} \, t\log t \qquad \textrm{for all }t,k\ge1.
\end{align}
Using $r=1.99$, we have $1/r-\log2 <-1/6$ so that
\begin{align}
I_1 := \int_{2^{k+1}}^{e^{e^{k/r}}} \frac{N_{k+1}(t)}{t^2\log t}\;dt \ \ll \ \frac{k^4}{2^k}\int_{2^{k+1}}^{e^{e^{k/r}}} \frac{dt}{t} \ < \ \frac{k^4}{2^k}e^{k/r}
\ \ll \ e^{-k/6}.
\end{align}

For the bulk of the integral, $I_2$, we apply Sathe--Selberg to get
\begin{align*}
I_2 := \int_{e^{e^{k/r}}}^\infty \frac{N_{k+1}(t)}{t^2\log t}\;dt \ & = \ \frac{1}{k!}\int_{e^{e^{k/r}}}^\infty G\Big(\frac{k}{\log\log x}\Big)\frac{(\log\log t)^k}{t(\log t)^2}\bigg(1 \ + \ O\Big(\frac{k}{(\log\log t)^2}\Big)\bigg)\;dt \\
& = \frac{1}{k!}\int_{k/r}^\infty G(k/y)y^ke^{-y}\big(1+O(k/y^2)\big)\;dy.
\end{align*}
Since $G(z)\ll 1$ for $z\le r$, the error in $I_2$ is bounded by
\begin{align*}
\frac{k}{k!}\int_{k/r}^\infty G(k/y)y^{k-2}e^{-y}\;dy \ll \frac{\Gamma(k-1)}{(k-1)!} = \frac{1}{k-1}
\end{align*}
so that
\begin{align}
I_2 \ = \ \frac{1}{k!}\int_{k/r}^\infty G(k/y)y^ke^{-y}\;dy \ + \ O(1/k).
\end{align}
The bulk of $I_2$ lies in the range $|y-k|<k^\delta$, for $\delta=1/2+\epsilon$, so we split accordingly $I_2 = J + J'$. Note $y^ke^{-k}$ is increasing for $y$ up to $k$, and decreasing thereafter. So letting $z=y- k$,
\begin{align*}
\frac{y^ke^{-y}}{k!} & < \exp(k\log y - k\log k - z) = \exp\big(k\log(1+z/k) - z\big) < e^{-z^2/2k}
\end{align*}
using the lower bound $k!> k^k e^{-k}$.
Thus
\begin{align}
\frac{y^ke^{-y}}{k!} < e^{-\tfrac{1}{2}k^{2\delta-1}} \qquad \text{for}\quad |y- k| > k^\delta.
\end{align}
and so, using $G(z)\ll 1$ again, the integral $J'$ is bounded by
\begin{align}
J' & := \frac{1}{k!}\bigg(\int_{k/r}^{k-k^\delta}+\int_{k+k^\delta}^\infty\bigg) G(k/y)y^ke^{-y}\;dy \ \ll \ e^{-k^{\epsilon}}.
\end{align}
For $|y- k| < k^\delta$, we have $G(k/y) = 1 \ + \ O\big(k^{\delta-1}\big)$  \cite[p. 236]{MVtext} so that
\begin{align}
J := \frac{1}{k!}\int_{k-k^\delta}^{k+k^\delta} G(k/y)y^ke^{-y}\;dy & = \frac{1+O(k^{\delta-1})}{k!}\int_{k-k^\delta}^{k+k^\delta} y^ke^{-y}\;dy \nonumber\\
& = \ \frac{1+O(k^{\delta-1})}{k!}\Gamma(k+1) \ = \ 1+O(k^{\delta-1}).
\end{align}
Combining altogether, we obtain $f(\N_k) = I_1+J+J' = 1 + O(k^{-1/2+\epsilon})$.

Similarly, the squarefree version of the Sathe--Selberg theorem \cite[p. 237, ex. 4]{MVtext} states
\begin{align}
N^*_{k+1}(x) \ = \ G^*\Big(\frac{k}{\log\log x}\Big)\frac{x}{\log x}\frac{(\log\log x)^k}{k!}\bigg(1 \ + \ O\Big(\frac{k}{(\log\log x)^2}\Big)\bigg)
\end{align}
where $G^*(z) = \frac{1}{\Gamma(1+z)}\prod_p(1+z/p)(1-1/p)^z$. And since $ G^*(1)=\prod_p(1-p^{-2})=6/\pi^2$, by an analogous argument we obtain $f(\N^*_k) = 6/\pi^2  + O\big(k^{-1/2+\epsilon}\big)$.
\end{proof}

Further, by similar arguments one may show for any choice of integers $e_p\ge0$ for each prime $p$,
\begin{align}
\sum_{\substack{\Omega(n)\,=\,k\\p^{e_p} \,\nmid\, n \;\forall e_p>0}}\frac{1}{n\log n} \ \sim \ \prod_{e_p>0} \big(1 - p^{-e_p}\big).
\end{align}

In particular, we deduce that the evens and odds asymptotically contribute equally $1/2$ to $f(\N_k)\sim 1$. Whereas in the squarefree case, the evens and odds contribute $2/\pi^2,4/\pi^2$, respectively, to $f(\N^*_k)\sim 6/\pi^2$.

Generalizing in another direction, we may consider the contribution to $f(\N_k)$ from an arbitrary arithmetic progression.
\begin{corollary}\label{cor:fNkq}
For any fixed integers $0\le a<q$, as $k\to\infty$
\begin{align}\label{eq:fNkq}
\sum_{\substack{\Omega(n)\,=\,k\\n\,\equiv\, a\;(q)}}\frac{1}{n\log n} \ = \ \frac{1+o_q(1)}{q}.
\end{align}
\end{corollary}
\begin{proof}
Consider the counting function $N_k(x;q,a) := \#\{n\le x: \Omega(n)=k,n\equiv a\;(q)\}$. First, if $(a,q)=1$, then Theorem 2 in Spiro \cite{Spiro} gives\footnote{Equation \eqref{eq:Spiro} may also be derived from earlier work of Delange \cite{Del}.}
\begin{align}\label{eq:Spiro}
N_{k+1}(x;q,a) \ = \ \frac{1}{\phi(q)} G_q\Big(\frac{k}{\log\log x}\Big)\frac{x}{\log x}\frac{(\log\log x)^k}{k!}\bigg(1 \ + \ O_q\Big(\frac{k}{(\log\log x)^2}\Big)\bigg)
\end{align}
for $k\le 1.99\log\log x$, where 
\begin{align*}
G_q(z) := \frac{(\phi(q)/q)^z}{\Gamma(1+z)}\prod_{p\nmid q} (1-z/p)^{-1}(1-1/p)^z = G(z)\prod_{p\mid q}(1-z/p).
\end{align*}
And since $G_q(1) = \phi(q)/q$, by a similar argument to Theorem \ref{thm:fNksim1} we obtain \eqref{eq:fNkq} (with error term $O_q(k^{-1/2+\epsilon})$).

When $d=(a,q)>1$, letting $i=\Omega(d)$, by the above argument we have
\begin{align*}
\sum_{\substack{\Omega(m)=k-i\\m\,\equiv\, a/d\;(q/d)}}\frac{1}{m\log m} \sim \frac{1}{q/d}.
\end{align*}
Thus \eqref{eq:fNkq} follows by setting $n=md$ and noting $\log md \sim \log m$.
\end{proof}
The fact that every progression $a\pmod{q}$ contributes $1/q$ to $f(\N_k)\sim1$ is especially remarkable in view of \eqref{eq:Spiro}, since $N_k(x;q,a)\sim N_k(x)/q$ is not always true. However, it does indeed hold in the critical range of $\log\log x\sim k$.

One may also obtain an analogous result in the squarefree case. Namely, under the same conditions as Corollary \ref{cor:fNkq}, if $d=(a,q)$ is squarefree then
\begin{align}
\sum_{\substack{\Omega(n)=k\\n\,\equiv\, a\;(q)}}\frac{\mu(n)^2}{n\log n} \ \sim \ \frac{1}{q}\prod_{p\nmid d}(1-p^{-2})\prod_{p\mid d}(1-p^{-1}) = \frac{6/\pi^2}{q}\frac{d}{\sigma(d)}.
\end{align}

In the next section, we turn to the question of optimal error bounds in Theorem \ref{thm:fNksim1}.

\section{Further progress via zeta functions for large $k$}

As with the results of Bayless et al. \cite{BKK}, Theorem \ref{thm:fNksim1} is proven by using partial summation and knowledge of the counting function for $\N_k$. And given the success of the zeta function approach for small $k$, one might hope that the approach would yield results that beat the stated bound of $O(k^{-1/2+\epsilon})$. Inspecting in Figure \ref{fig:fNk2}, the ratio of consecutive entries of $1-f(\N_k)$ appears to converge to $1/2$, suggesting that the true error term in Theorem \ref{thm:fNksim1} may be $O(2^{-k})$. Such a bound remains out of reach for the moment, but we obtain related partial progress in this direction, culminating with the {\bf exponentially-decaying} upper bound $f(\N_k)\le 1 + O(k\,2^{-k/2})$.

\begin{theorem}\label{thm:intlogzetak}
We have
\begin{align}
\frac{1}{k!}\int_1^\infty [\log\zeta(s)]^k\;ds \ = \ 1 \ + \ O(2^{-k}).
\end{align}
\end{theorem}
\begin{proof}
Fix $k$ sufficiently large. The result will follow from the following three claims,
\begin{align}
\int_1^2 \log(\tfrac{1}{s-1})^k\;ds \ = \ k!, \qquad  0 < \int_2^\infty [\log\zeta(s)]^k\;ds \ < \ 2^{-k},
\end{align}
and
\begin{align}
0 \ < \ J_k \ \ll \ 2^{-k} \qquad\quad \textrm{where}\quad J_k:=\frac{1}{k!}\int_1^2 [\log\zeta(s)]^k - \log(\tfrac{1}{s-1})^k\;ds.
\end{align}

First, we have
\begin{align}
\int_1^2 \log(\tfrac{1}{s-1})^k\;ds = \int_0^1 (-\log s)^k\;ds = \int_0^\infty u^k e^{-u}\;du 
= k!\,.
\end{align}

Second, we note $0<\log\zeta(s)<2^{1-s}$ for $s\ge2$, so
\begin{align}
\int_2^\infty \log\zeta(s)^k\;ds < \int_1^\infty 2^{-ks}\;ds = \frac{2^{-k}}{k\log 2} < 2^{-k}.
\end{align}

Third, the series expansion of $\log\zeta$ at $s=1$ is
\begin{align*}
\log\zeta(s) = \log\big(\tfrac{1}{s-1}\big) + \gamma(s-1) + O\big((s-1)^2\big).
\end{align*}
It will suffice to use $0<\log\zeta(s) +\log(s-1) < .6(s-1)$ for $s\in[1,2]$. 
Expanding the binomial gives
\begin{align*}
[\log\zeta(s)]^k < \Big(.6(s-1)-\log(s-1)\Big)^k = \sum_{j=0}^k \binom{k}{j}[.6(s-1)]^j[-\log(s-1)]^{k-j}
\end{align*}
so that
\begin{align}
J_k := \frac{1}{k!}\int_1^2 [\log\zeta(s)]^k - [-\log(s-1)]^k\;ds < \sum_{j=1}^k \frac{.6^j}{j!(k-j)!}\int_0^1 s^j[-\log s]^{k-j}\;ds.
\end{align}
Note we have
\begin{align}\label{eq:intcombo}
\int_0^1 s^j[-\log s]^i\;ds = \int_0^\infty u^i\,e^{-(j+1)u}\;du 
= \frac{i!}{(j+1)^{i+1}}.
\end{align}
Hence using $i=k-j$, we bound $J_k$ by the hypergeometric series
\begin{align}
J_k < \sum_{j=1}^{k} \frac{.6^j}{j!}(j+1)^{-(k-j+1)} \le .6\cdot 2^{-k}\Big(1 + \sum_{j=2}^{k}.6^{j-1}\Big) \ll 2^{-k},
\end{align}
since the ratio of consecutive terms is $.6\big(\frac{j+1}{j+2}\big)^{k-j}\le.6$ for $j\ge 2$.
\end{proof}

Recall the prime zeta function $P(s) = \sum_p p^{-s}$. Consider
\begin{align*}
h(s) := \log\zeta(s) - P(s) = \sum_{m\ge 2}\sum_p \frac{1}{mp^{ms}} \ = \ h(1) + h'(1)(s-1) + O\big((s-1)^2\big)
\end{align*}
where the Taylor series coefficients are given by
\begin{align*}
c := h(1) & = \sum_{m\ge 2}\sum_p \frac{1}{mp^{m}} =  .315718\cdots \quad\textrm{ and}\\
h'(1) & = -\sum_{m\ge 2}\sum_p \frac{\log p}{p^{ms}}\Big|_{s=1}
= -\sum_p \frac{\log p}{p(p-1)} = -.7505\cdots.
\end{align*}
Hence from the expansion $P(s)+h(s)=\log\zeta(s) = \log(\tfrac{1}{s-1})+\gamma(s-1)+O(s-1)^2$, we obtain
\begin{align}
P(s) & = \log\big(\tfrac{1}{s-1}\big) - c + (\gamma-h'(1))(s-1) + O\big((s-1)^2\big).
\end{align}
In particular, we have $0<P(s) - \log\big(\tfrac{\alpha}{s-1}\big)<1.4(s-1)$ for $s\in[1,2]$, where
\begin{align}\label{eq:alpha}
\alpha := e^{-c} = \exp\Big(-\sum_{m\ge2}\frac{P(m)}{m}\Big) = \prod_{m\ge2}\zeta(m)^{\mu(m)/m}=.729264\cdots.
\end{align}
Proceeding as in Theorem \ref{thm:intlogzetak}, we obtain
\begin{theorem}\label{thm:Pkint}
For $\alpha$ in \eqref{eq:alpha}, we have
\begin{align}
\frac{1}{k!}\int_1^\infty P(s)^k\;ds \ = \ \alpha \ + \ O(2^{-k}).
\end{align}
\end{theorem}
\begin{proof}
The proof follows from the following three claims,
\begin{align}
\int_1^2 \log(\tfrac{\alpha}{s-1})^k\;ds \ = \ \alpha\, k!, \qquad  0 < \int_2^\infty P(s)^k\;ds < 2^{-k},
\end{align}
and
\begin{align}
0 < I_k \ \ll \ 2^{-k} \qquad\quad \textrm{for}\quad I_k:=\frac{1}{k!}\int_1^2 P(s)^k - \log(\tfrac{\alpha}{s-1})^k\;ds.
\end{align}
The first two hold, as in Theorem \ref{thm:intlogzetak}, recalling $P(s) < \log\zeta(s)$. For the third, as in Theorem \ref{thm:intlogzetak}, using $0<P(s) - \log\big(\tfrac{\alpha}{s-1}\big)<1.4(s-1)$ and expanding the binomial gives
\begin{align*}
I_k < \sum_{j=1}^{k} \frac{1.4^j}{j!}(j+1)^{-(k-j+1)} \ <  1.4\cdot2^{-k}+\frac{1.4^2}{2!}3^{1-k}k \ \ll \ 2^{-k}.
\end{align*}
Here we used the fact that the terms in the series decrease for $j$ up to around $k/\log k$ and increase thereafter. Thus the series is bounded by the first term, plus $k$ times the max of the terms $j=2,k$.
\end{proof}

These results culminate with the following exponentially-decaying upper bound for $f(\N_k)$.

\begin{theorem}\label{thm:Pk2k}
We have
\begin{align}
f(\N_k) = \int_1^\infty P_k(s)\;ds \ \le \ 1 \ + \ O(k\,2^{-k/2})
\end{align}
\end{theorem}
\begin{proof}
We first apply the combinatorial identity \eqref{eq:Pkpart},
\begin{align*}
f(\N_k) = \int_1^\infty P_k(s)\;ds \le \sum_{n_1 + 2n_2 +\cdots = k}\prod_{j\ge2} \frac{1}{n_j!}\big(P(j)/j\big)^{n_j}\int_1^\infty \frac{P(s)^{n_1}}{n_1!} \;ds
\end{align*}
using $P(js)\le P(j)$ for $j\ge2, s\ge1$. Then by Theorem \ref{thm:Pkint},
\begin{align*}
f(\N_k) \ \le \sum_{n_1 + 2n_2 +\cdots = k}\prod_{j\ge2} \frac{1}{n_j!}\big(P(j)/j\big)^{n_j}\big(\alpha + O(2^{-n_1})\big) = \sum_{n_1=0}^k Q(k-n_1)\big(\alpha + O(2^{-n_1})\big)
\end{align*}
where $Q(m) := \sum_{2n_2 + 3n_3+\cdots = m}\prod_{j\ge2} \frac{1}{n_j!}\big(P(j)/j\big)^{n_j}$. Then by definition of $\alpha,c$ from \eqref{eq:alpha},
\begin{align}\label{eq:alphaQ}
\frac{1}{\alpha} = e^c = \exp\Big(\sum_{j\ge 2}P(j)/j\Big) = \sum_{m\ge0}\sum_{2n_2 + 3n_3+\cdots = m}\prod_{j\ge2} \frac{1}{n_j!}\big(P(j)/j\big)^{n_j} =\sum_{m\ge0}Q(m)
\end{align}
By Lemma \ref{lma:Qk} below, we have $\sum_{m\ge y}Q(m) \ll y\,2^{-y}$. Hence using $y=k/2$,
\begin{align*}
f(\N_k) \le \sum_{m=0}^k Q(m)\big(\alpha + O(2^{m-k})\big) & \le \big(\alpha + O(2^{y-k})\big)\sum_{0\le m\le y} Q(m) + O(y\,2^{-y})\\
& \le 1 \ + \ O(k\,2^{-k/2})
\end{align*}
by \eqref{eq:alphaQ}. This completes the proof.
\end{proof}

\begin{lemma}\label{lma:Qk}
We have
\begin{align}
Q(k)=\sum_{2n_2 \cdots = m}\prod_{j\ge2} \frac{1}{n_j!}\big(P(j)/j\big)^{n_j} \ \ll \ k\,2^{-k}.
\end{align}
\end{lemma}
\begin{proof}
First $P(2) < 2\cdot 2^{-2}$ implies $P(j) < 2\cdot 2^{-j}$ for $j\ge2$, and so
\begin{align*}
Q(k) < \sum_{2n_2+3n_3\cdots = k}\prod_{j\ge2} \frac{1}{n_j!}\big(2\cdot 2^{-j}/j\big)^{n_j} =: \widetilde Q(k).
\end{align*}
Then we have a power series identity
\begin{align*}
\sum_{k\ge0}\widetilde Q(k) z^{k} &= \sum_{k\ge0}\sum_{2n_2+3n_3\cdots = k}\prod_{j\ge2} \frac{1}{n_j!}\big(2\cdot(z/2)^j/j\big)^{n_j} = \prod_{j\ge2}\sum_{n_j\ge0}\frac{1}{n_j!}\big(2\cdot(z/2)^j/j\big)^{n_j}\\
& = \prod_{j\ge2}\exp\Big(2\cdot(z/2)^j/j\Big) = \exp\Big(2\sum_{j\ge2}\frac{(z/2)^j}{j}\Big) = \exp\Big(z - 2\log(1-z/2)\Big)\\
& = e^z\cdot (1-z/2)^{-2} = \sum_{n\ge0} \frac{z^n}{n!} \sum_{m\ge0} (m+1)(z/2)^m.
\end{align*}
Hence equating coefficients of $z^k$, we deduce
\begin{align}
\widetilde Q(k) = \sum_{n+m=k}\frac{1}{n!}\frac{m+1}{2^m} = 2^{-k}\sum_{n=0}^k(k-n+1)\frac{2^n}{n!} < k2^{-k} e^2 \ll k2^{-k}.
\end{align}
This completes the proof.
\end{proof}

While a corresponding lower bound to $f(\N_k)$ as in Theorem \ref{thm:Pk2k} remains out of reach, the zeta function method does produce strong enough lower bounds to show $f(\N_k)$ is minimal when $k=6$.

\subsection{A global minimum via zeta functions}
In view of Proposition \ref{prop:Pkpart}, the integral $(1/k!)\int_1^\infty P(s)^k\;ds$ in Theorem \ref{thm:Pkint} is a lower bound for $f(\N_k) = \int_1^\infty P_k(s)\;ds$, and constitutes the first of the terms in the identity \eqref{eq:Pkpart}, one per partition of $k$. Note the terms of partitions built from small parts contribute the most. So by incorporating the terms for the partitions $k = 1\cdot(k-j) + j$ and $k = 1\cdot(k-j-2) + 2 + j$ for $j\le 6$, we may obtain a sufficiently tight lower bound on $f(\N_k)$ to conclude the following.

\begin{theorem}
We have $f(\N_6) < f(\N_k)$ for all positive integers $k\neq6$.
\end{theorem}
\begin{proof}
We have already verified the claim directly for $k\le 20$, see Figures \ref{fig:fNk},\ref{fig:fNk2}. Thus it suffices to assume $k>20$. By Proposition \ref{prop:Pkpart},
\begin{align}\label{eq:Pkcombshort}
f(\N_k) = \int_1^\infty P_k(s)\;ds \ & > \ \frac{1}{k!}\int_1^2 P(s)^k\;ds + \sum_{j=2}^6\frac{\int_1^2 P(s)^{k-j}P(js)\;ds}{j(k-j)!} \nonumber\\
& \quad + \ \frac{\int_1^2 P(s)^{k-4}P(2s)^2\;ds}{2!2^2(k-4)!} + \sum_{j=3}^6\frac{\int_1^2 P(s)^{k-j-2}P(2s)P(js)\;ds}{2j(k-j-2)!}.
\end{align}
For every $k\ge1$, we have
\begin{align}
\int_1^\infty P(s)^k\;ds > \int_0^1 \log\big(\tfrac{\alpha}{s}\big)^k\;ds = \alpha\,k!\,.
\end{align}
Using the first order Taylor series $P(js) > P(j) + P'(j)(s-1)$ for $j\ge2$,
\begin{align*}
\int_1^2 P(s)^{k-j}P(js)\;ds \ > & \ P(j)\int_0^1 \log\big(\tfrac{\alpha}{s}\big)^{k-j}\;ds+P'(j)\int_0^1 \log\big(\tfrac{\alpha}{s}\big)^{k-j}s\;ds\\
& = \alpha(k-j)!\Big(P(j) + \frac{\alpha P'(j)}{2^{k-j}}\Big)
\end{align*}
and
\begin{align*}
\int_1^2 P(s)^{k-j-2}& P(2s)P(js) \;ds \ > \ (P'(2)P(j)+P(2)P'(j))\int_0^1 \log\big(\tfrac{\alpha}{s}\big)^{k-j}s\;ds\\
& \qquad + \ P(2)P(j)\int_0^1 \log\big(\tfrac{\alpha}{s}\big)^{k-j-2}\;ds \ + \  P'(2)P'(j)\int_0^1 \log\big(\tfrac{\alpha}{s}\big)^{k-j-2}s^2\;ds\\
& = \ \alpha(k-j-2)!\Big(P(2)P(j) + \frac{\alpha}{2^{k-j-1}}[P'(2)P(j)+P(2)P'(j)] + \frac{\alpha^2 P'(2)P'(j)}{3^{k-j-1}}\Big).
\end{align*}
Hence plugging back into \eqref{eq:Pkcombshort},
\begin{align}
f(\N_k) & > \alpha\bigg[1 + \sum_{j=2}^6 \frac{P(j)+\alpha P'(j)/2^{k-j}}{j} \ + \ \frac{1}{8}[P(2)^2+ \frac{\alpha P(2)P'(2)}{2^{k-4}} + \frac{\alpha^2 P'(2)^2}{3^{k-3}}] \nonumber\\
& \qquad \ + \
\sum_{j=3}^6 \frac{1}{2j}\bigg(P(2)P(j) + \frac{\alpha}{2^{k-j-1}}[P'(2)P(j)+P(2)P'(j)] + \frac{\alpha^2 P'(2)P'(j)}{3^{k-j-1}}\bigg)\bigg] =: \beta_k.
\end{align}
Finally, since the lower bound $\beta_k$ is clearly increasing in $k$, for $k>20$ we obtain\footnote{In Mathematica, we compute $\beta_{20} = 0.991049\cdots.$}
\begin{align}
f(\N_k) > \beta_k > \beta_{20} > .99 > f(\N_6).
\end{align}
This completes the proof.
\end{proof}

\subsection{Remark on integration}
The integral $f(\N_k)=\int_1^\infty P_k(s)\;ds$ is computed numerically with high confidence. Nevertheless, in this section, we supplement the data by rigorously bounding the integral $f(\N_k)=\int_1^\infty P_k(s)\;ds$ at the tail and near the singularity $s=1$. For $k\ge2$, the tail $s\ge 10$ contributes
\begin{align*}
\int_{10}^\infty P_k(s)\;ds \le 2\int_{10}^\infty 2^{-ks}\;ds = \frac{2^{1-10k}}{k\log 2} < 2^{-10k}
\end{align*}
Next for $\epsilon>0$,
\begin{align*}
\int_1^{1+\epsilon}\log(\tfrac{1}{s-1})^k\;ds = \int_{\log(1/\epsilon)}^\infty u^k e^{-u}\;du
\end{align*}
And since $P_k(s) \le \log(\tfrac{1}{s-1})^k$ for $s\in(1,1+e^{-2k})$ by Lemma \ref{lma:int1} below, letting $\epsilon_k=e^{-4k}$,
\begin{align*}
\int_1^{1+\epsilon_k} P_k(s)\;ds < \frac{1}{k!}\int_{4k}^\infty u^k e^{-u}\;du = \frac{\Gamma(k+1,4k)}{k!}
\end{align*}
for the incomplete Gamma function $\Gamma(k,z)=\int_z^\infty u^{k-1} e^{-u}\;du$, which is rapidly computable.

\begin{lemma}\label{lma:int1}
For all $k\ge1$, $s\in(1,1+e^{-2k})$, we have 
\begin{align}
P_k(s) \ \le \ \frac{1}{k!}\log(\tfrac{1}{s-1})^k.
\end{align}
\end{lemma}
\begin{proof}
We proceed by induction on $k$. For $k=1$, $P(s) \le \log(\tfrac{1}{s-1})$ holds for all $s>1$. Now assuming the claim holds for all $j<k$, the recursion \eqref{eq:Pkrecur} gives
\begin{align*}
P_k(s) = \frac{1}{k}\sum_{j=1}^k P(js)P_{k-j}(s) \le \frac{1}{k}\sum_{j=1}^k \log\big(\tfrac{1}{js-1}\big)\frac{1}{(k-j)!}\log\big(\tfrac{1}{s-1}\big)^{k-j} \le \frac{1}{k!}\log\big(\tfrac{1}{s-1}\big)^k
\end{align*}
which completes the induction, provided $\log\big(\tfrac{1}{js-1}\big) \le \frac{(k-j)!}{k!}\log\big(\tfrac{1}{s-1}\big)^j$. This is equality for $j=1$, and for $j\ge2$ it suffices to prove
\begin{align}\label{eq:logdom}
\log\left(\frac{1}{(j+1)s-1}\right) \le \frac{1}{k-j}\log\big(\tfrac{1}{s-1}\big)\log\left(\frac{1}{js-1}\right) \qquad 2\le j<k, s\in(1,1+e^{-2k}).
\end{align}
Letting $z=js-1$, we seek $s$ so that $\frac{1}{k-j}\log\big(\tfrac{1}{s-1}\big)$ is at least
\begin{align*}
\frac{\log(z+s)}{\log z} = 1 + \log(1+\tfrac{s}{z}) \le 1+\frac{s}{z} = 1+\frac{1}{j-1/s} \le 2
\end{align*}
using $j\ge2, s\ge1$. Hence \eqref{eq:logdom} holds since $s-1<e^{-2k}$.
\end{proof}

\subsection{Open questions}
From the data in Figures \ref{fig:fNk},\ref{fig:fNk2} one might suspect that $f(\N_k)$ increases monotonically to 1 for $k\ge6$, while $f(\N^*_k)$ decreases monotonically to $6/\pi^2$ for all $k\ge1$. Similar numerical experiments suggest that $f(\mathbb O_k)$ also decreases monotonically to $1/2$ for $k\ge1$, where $\mathbb O_k$ are the odd members of $\N_k$.

\section*{Acknowledgments}
The author is grateful to Carl Pomerance for many helpful discussions, particularly regarding the proof of Theorem \ref{thm:fNksim1}. The author would also like to thank Paul Kinlaw for his input and verification of the computations in Figure \ref{fig:fNk}, and to the referee for useful comments. The author is supported by a Churchill Scholarship at the University of Cambridge and a Clarendon Scholarship at the University of Oxford.

\bibliographystyle{amsplain}

\end{document}